\documentclass[leqno,12pt]{article} 
\usepackage[dvipdfmx,colorlinks=true]{hyperref}
\hypersetup{urlcolor=blue, citecolor=red}
\usepackage{amsmath, amssymb, bm}
\usepackage{amsthm} 
\theoremstyle{plain} 
\newtheorem{theorem}{Theorem}[section] %
\newtheorem{lemma}[theorem]{Lemma}

\newtheorem{proposition}[theorem]{Proposition}

\theoremstyle{definition} %

\newtheorem{remark}[theorem]{Remark}

\newcommand{\R}{\mathbb{R} }

\newcommand{\LB}{\left \lbrack }
\newcommand{\RB}{\right  \rbrack }
\newcommand{\LC}{\left ( }
\newcommand{\RC}{\right ) }
\newcommand{\LD}{\left \{ }
\newcommand{\RD}{\right \} }

\newcommand{\DS}{\displaystyle }

\DeclareMathOperator{\esssup}{ess\sup}

\DeclareMathOperator{\diag}{diag}
\DeclareMathOperator{\id}{id}
\DeclareMathOperator*{\argmax}{argmax} 
\DeclareMathOperator*{\argmin}{argmin}

\DeclareMathOperator*{\conv}{conv}
\DeclareMathOperator*{\Proj}{Proj}

\makeatletter
\def\address#1#2{\begingroup
\noindent\parbox[t]{7.8cm}{%
\small{\scshape\ignorespaces#1}\par\vskip1ex
\noindent\small{\itshape E-mail address}%
\/: #2\par\vskip4ex}\hfill%
\endgroup}%
\makeatother
\title{\uppercase{Nonlinear Evolution Equation Associated with Hypergraph Laplacian}} %
\author{Masahiro Ikeda, 
Shun Uchida\footnote{Corresponding author: shunuchida@oita-u.ac.jp}
\date{} %
}

\begin{document}

\maketitle

%

\begin{abstract}
Let 
$V$ be a finite set, $E \subset 2^{V} $ be a set of hyperedges, and $w : E \to (0, \infty)$ be an edge weight.
On the (wighted) hypergraph  $G = (V ,E ,w )$,
we can define 
a multivalued nonlinear operator $L_{G,p}$ ($p \in [1 ,\infty )$) 
as the subdifferential of a convex function on $\R ^V $, 
which is called ``hypergraph $p$-Laplacian.''
In this article,
we first introduce an inequality for this operator $L_{G,p}$
which resembles the Poincar\'{e}-Wirtinger inequality in PDEs.
Next we consider an ordinary differential equation on $\R ^V $
governed by $L_{G,p}$, which
is referred as ``heat'' equation on the hypergraph and used to study the
geometric structure of graph in recent researches.
With the aid of the Poincar\'{e}-Wirtinger type inequality,
we can discuss the existence and 
the large time behavior of solutions to the ODE
by procedures similar to those for 
the standard heat equation in PDEs with the zero Neumann boundary condition.
\end{abstract}

\section{Introduction}
Weighted hypergraph $G $ is the triplet of 
a finite set $V $ (vertex set),
a family $E \subset 2^{V} $ of subsets with more than one element of $V$ (set of hyperedges),
and a function  $w : E \to (0, \infty)$ (edge weight).
In this paper, 
we consider the so-called hypergraph $p$-Laplacian $L _{G , p }:\R^V \to 2^{\R^V}$ defined on
 $G=(V, E ,w)$ with $p\in [1,\infty)$ (see \S \ref{defLap} for the precise definition) and 
the following ordinary differential equation associated with $L _{G , p }$:
\begin{equation}
 \frac{d}{dt} x (t) + L_{G,p }(x (t) ) \ni h (t)    ,
\label{Eq} 
\end{equation}
where $x : [0,T] \to \R ^V $ is an unknown function and
$h : [0,T] \to \R ^V $ is a given external force with $T>0$.
This equation is referred as a ``heat'' equation on the hypergraph.

When $G$ is a usual graph,
namely, if $E $ consists of subsets  with two elements of $V$, then $L _{G , 2 }$
 becomes single-valued and coincides with a 
square matrix $D -A $,
where
$D$ and $A$ stand for the degree matrix and the adjacency matrix of $G$, respectively (see Remark \ref{Usual-Garph} below).
Then the random walk on the graph
can be characterized by $ L _{G ,2 } D ^{-1} = I - A D ^{-1}$,
called the random walk normalized Laplacian matrix on  $G$.
In this case,
the weight function $w $ in the Laplacian
can be regarded as the conductance (reciprocal of resistance)
of ``current'' or ``flow'' over each edge $e \in E$.
Such a matrix can be used to investigate some properties of network represented by a graph.
In particular, one of the important applications is PageRank, 
an algorithm to determine the importance of a Web site
introduced  by Brin--Page \cite{B-P}.
Moreover, the linear differential equation $x'(t) =  (I - A D ^{-1}) x (t)$
can be found in the definition of another pagerank, called the heat kernel pagerank given by  Chung \cite{Chung}.
Such a ``heat'' equation is also applied in \cite{Chung} to show a Cheeger type inequality,  a relation between the first positive eigenvalue of the normalized Laplacian $ L _{G ,2 } D ^{-1}$
and the geometric structure of graph.
 See also \cite{Chung-book} and references therein for more information of Cheeger inequalities on graphs.

As for the case where $G$ is a hypergraph 
(i.e., the family $E$ possibly possesses some sets with more than two elements), 
we can find various applications 
to, e.g.,  neural network \cite{Neu} or molecular modeling in chemistry \cite{Mole}.
In order to study geometric properties in the sense of Cheeger's inequality,
Louis \cite{Louis} introduced a generalized Laplacian 
and consider a heat equation defined on hypergraph (see also \cite{Jost01}\cite{Jost02}).
A variant definition is given by Yoshida \cite{Yoshida00},
which comprises subdifferentials of Lov\'{a}sz extension of submodular transformation. 
This type of Laplacian on hypergraph, which is the main target of this paper,
is often used in recent researches, e.g, \cite{FSY}\cite{IKT}\cite{L-H-M}\cite{TMIY}
and generalized to $p$-Laplacian in \cite{1-Lap} and \cite{L-M}. 
Remark that the hypergraph $p$-Laplacian $L_{G , p}$ defined in this manner
becomes a multivalued operator.

The main purpose of this paper is to 
state some basic tools and usages of the hypergraph $p$-Laplacian $L _{G, p}$ introduced by Yoshida \cite{Yoshida00} 
 from a viewpoint of the nonlinear evolution equation theory.
In the next section, we give a precise definition of $L _{G, p}$
and check that  $L _{G, p}$ can be written as a subdifferential of a functional $\varphi _{G, p}$ (see Proposition \ref{Maximal-monotone}).
Although this fact has already been pointed out in several articles,
we shall give a proof for self-containment.
We also show that $L _{G, p}$ satisfies a Poincar\'{e}-Wirtinger type inequality
and the structure of $L _{G, p}$ is quite similar to the standard Neumann Laplacian in PDE.
By using these facts, we deal with the Cauchy problem and the time-periodic problem of \eqref{Eq} in Section 3.
Since one can see that $L _{G, p}$ coincides with the subdifferential of a convex functional,
the K\={o}mura--Br\'{e}zis theory is applicable and
the existence of a unique global solution  to Cauchy problem of  \eqref{Eq} can be assured.
From  Poincar\'{e}'s inequality,
we shall derive a decay estimate of solution to \eqref{Eq} with $h \equiv 0$.
On the other hand,
we can not apply the abstract theory for the  time-periodic problem
since the coercivity of  $\varphi _{G, p}$ does not hold (see Theorem \ref{0-eigen-Rev} below).
We here  employ a technique for parabolic equations governed by the homogeneous Neumann Laplacian (see, e.g., \cite{O-U})
and assure the existence of periodic solution to \eqref{Eq}.

\section{Properties of Hypergraph $p$-Laplacian}
\subsection{Preliminary}
We first fix some
terms and definitions of maximal monotone operator and subdifferential operator
(see, e.g., \cite{Bar}\cite{Bre}\cite{Show}).
Let $H$ be a real Hilbert space with the norm $|\cdot |  $ and the inner product $(\cdot ,\cdot ) $
and $A$ be a (possibly) multivalued operator from $H$ into $2 ^H$, which stands for the power set of $H$.
The domain and range of $A$ are denoted by $D(A)$ and  $R(A)$, respectively.
An operator $A$ is said to be monotone if
$ (y _1 - y_2 ,  x_1  - x_2 ) \geq 0$ holds for any $x _j \in D(A)$ and $y_j \in A x_j $ ($j=1,2$)
and a monotone operator $A$ is said to be maximal monotone if
$ R(\id + A) = H$, where $\id $ is the identity map.
It is well known that
\begin{itemize}
\item If $A $ is maximal monotone, then $A x $ becomes a closed convex subset in $H$ for any $x \in D(A)$.  
Based on this fact, we define the minimal section of $A $
by $A ^{\circ } x := (Ax )^{\circ }$, where $C ^{\circ } := \argmin _{y \in C } |y | = \Proj _ {C} 0 $
for a closed convex set $C \subset H$.
 
\item The maximal monotone operator is demiclosed.
That is,
if $y _m \in A x_m$, $x _ m \to x $ strongly in $H$, and $y _ m \rightharpoonup  y $ weakly in $H$ as $m\to \infty$,
then the limits satisfy $x \in D(A)$  and $y \in A x$.
\end{itemize}

Let $\varphi : H \to ( -\infty , + \infty ]$ be a proper (i.e., $\varphi \not \equiv + \infty $) 
lower semi-continuous (l.s.c., for short) and convex functional.
The set $D(\varphi ) := \{ x \in H ;~\varphi (x) < +\infty \}$ is called the effective domain of $\varphi $.
Then we can define a (possibly) nonlinear multivalued mapping on $H$ by 
\begin{equation*}
\partial \varphi  :  x \mapsto \{ \eta \in H;~~(\eta , \xi - x )
		\leq \varphi (\xi ) -\varphi (x )~~\forall \xi \in D(\varphi )  \} ,
\end{equation*}
known as the subdifferential of $\varphi $.
As for the basic result in the convex analysis, we can see that
\begin{itemize}
\item The convex function $\varphi $ is continuous at the interior points of $D(\varphi )$.
Especially, if $\varphi $ is defined on the whole space $H$ (i.e., $D(\varphi ) = H $), 
then $\varphi $ is continuous on $H$.

\item The subdifferntial of a proper l.s.c. convex functional is always maximal monotone.

\item  Let $\varphi  ,\psi : H \to (  -\infty , + \infty ] $ be proper l.s.c. convex functions
such that the intersection of $D(\varphi )$ and the interior of $D(\psi )$ is not empty.
Then it follows that $ \partial  (\varphi  + \psi ) =  \partial  \varphi  + \partial  \psi $.

\end{itemize}
Moreover, 
assume that $\varphi $ is even, namely
$D(\varphi )$  is symmetric (i.e., $x \in D(\varphi )$ iff $- x \in D(\varphi )$) and $\varphi (- x ) = \varphi (x)$.
Then the subdifferential $\partial \varphi $ becomes  an odd operator,
that is,
$D(\partial \varphi )$ is symmetric and $\partial \varphi (- x ) = -\partial \varphi (x)$.
Indeed, if $\eta \in \partial \varphi (x )$,
\begin{align*}
& ~~( \eta , \xi -  x ) \leq \varphi (\xi ) -\varphi (x )~~\forall \xi \in D(\varphi ) \\
\Leftrightarrow ~~& ~~
(- \eta , -\xi  - (- x) ) \leq \varphi (-\xi ) -\varphi (-x )~~\forall \xi \in D(\varphi ) ,
\end{align*}
which implies $- \eta \in \partial \varphi (-x)$.

\subsection{Definition of Hypergraph $p$-Laplacian}
\label{defLap}
Let $V := \{ v _1 ,  \ldots ,  v_ n \}$ be a vertex set and 
$E \subset 2 ^ V  $ (the power set of $V$) be a set of hyperedges. 
Note that each $e \in E $ consists of more than one element of $V$. 
Moreover, a positive function  $w : E \to (0, \infty ) $ is defined as a weight on each hyperedge $e \in E$.
Then the triplet $G = (V , E , w)$ is called the (weighted) hypergraph.

In this paper, we shall consider some ordinary differential equations over $ \R ^ V $, 
the set of functions $x : V \to \R $.
By letting $x_ i : = x (v_ i) $,
we can identify $\R ^V $ with the $n$-dimensional Euclidean space $\R ^n $.
Define norms on $\R ^V $ by 
\begin{equation*}
|x | _{\ell ^q }
:=\begin{cases}
~~ \DS \LC \sum_{v \in V} |x(v)| ^q \RC ^{1/q} ~~&~~\text{ if } q \in [1 , \infty ) ,\\  
~~ \DS \max _{v \in V} |x(v)|  ~~&~~\text{ if } q = \infty  ,
\end{cases}
\end{equation*}
and the $\ell ^2$-inner product on $\R ^V $ by 
$x\cdot y := \sum_{v \in V } x(v) y(v) $
for each $x,y \in \R^ V $.
Recall that the standard inequality 
$  |x | _{\ell ^r } \leq  |x | _{\ell ^q } $ holds if $ q < r $.

Henceforth, we express the indicator function on  $S \subset V$ by   
$1 _S  \in \R^ V $, i.e., 
let $1 _  S (u) = 1$ if $u \in S$ and $1 _ S (u) = 0$ if $u \not \in S$.
Then we define the base polytope for the hyperedge $e \in E$  by
\begin{equation*}
B_e := \conv  \{ 1 _u - 1 _ v ;~~~u,v \in e  \},
\end{equation*}
where 
 $1 _ v := 1 _{ \{ v \} }$ and
$\conv Q$ denotes the convex hull of $Q \subset \R ^V $. 

We here define 
\begin{equation*}
 f_e (x) := \max _{u,v \in e} ( x (u ) - x(v))= \max _{u,v \in e} | x (u ) - x(v)|
= \max _{b \in B_e } b \cdot x.
\end{equation*}
Obviously, we have 
\begin{itemize}
\item $f_e (x ) \geq 0 $ and  $f_e (x )  = f_e (-x ) $ for any $x \in \R ^V$.  
\item $f_e (x) = 0 $ iff $ x(u) = x(v)$ for any $u,v \in e$.
\end{itemize}
Moreover, since $f _e $ is convex and its domain coincides with the whole space $\R ^V $
(which yields the continuity of $f_e $), we can define the subdifferential of $f_e $.
Here we recall the following maximum rule of subdifferential (see, e.g., Proposition 2.54 in \cite{M-N}):
\begin{lemma}
\label{MaxLem} 
Let $g _ j : \R ^ V  \to \R $ ($j =1 ,2, \ldots ,m$) be convex functions satisfying $D(g _ j) = \R ^ V  $.
Then for every $x \in \R ^V$ it holds that 
\begin{equation*}
\partial ( \max _{k} g_k (x) ) = \conv \LC \bigcup _{j \in J (x)} \partial g_ j (x) \RC ,
\end{equation*}
where 
\begin{equation*}
J (x):= \left\{ j \in \{  1, 2, \ldots , m \} ;~~g _j (x) = \max _  k g _k (x) \right\} .
\end{equation*}
\end{lemma}
Thanks to Lemma \ref{MaxLem}, the subdifferential of $f_e $ can be represented by
\begin{equation*}
\partial f_ e  (x)  =
\argmax _{b \in B_e} b \cdot x  = \left\{ b_e  \in B_e ;~~b_e \cdot x = \max _{b \in B_e} b \cdot x \right\}  .
\end{equation*}
By the standard argument of convex analysis,
we can assure the subdifferentiability 
of the following functional composed by the composition and sum of $f_e $:
\begin{proposition}
\label{Maximal-monotone} 
Let $g : [0 ,\infty ) \to \R $ be a convex, non-decreasing, and $C^1$-function.
Then 
\begin{equation*}
\varphi _{G , g} (x ) :=  \sum_{e \in E } w (e ) g ( f_e (x) ) 
\end{equation*}
is continuous, convex and even, and 
its domain $D (\varphi _{G ,g })$ coincides with $\R  ^V $.
Moreover, its subdifferential coincides with
\begin{align*}
L _{G , g } (x) &:=
\partial \varphi _{G , g}  (x) =
 \sum_{e \in E } w (e ) g' ( f_e (x) ) \partial   f_e (x) \\
&=\LD \sum_{e \in E } w(e) g' (f_e (x) )  b_e ;~~~b_e \in \argmax _{b \in B_e } b \cdot x  \RD ,
\end{align*}
which is an odd maximal monotone operator 
satisfying $D (L _{G , g }) = \R  ^V $.
\end{proposition}
\begin{proof}
This result has already shown by Corollary 3.5 of \cite{CLT} within a fairly general setting.
For the sake of completeness, however,  we here give another proof via Brouwer's fix point theorem.

Evidently, $D( \varphi _{G , g}) = \R ^ V$ and $\varphi _{G , g}$ is even and convex.
We here only demonstrate $\partial \varphi _{G , g} (x ) = L _{G , g } (x)$.
For any $x,y \in \R^ V $ and $ b_e \in \argmax _{b \in B_e } b \cdot x $, we get 
\begin{align*}
&\LC \sum_{e \in E } w(e) g' (f_e (x) )  b_e \RC \cdot  (y -x ) \\
 \leq & 
	 \sum_{e \in E } w(e) g' (f_e (x) )   (f_e (y)  - f_e (x) )  \\ 
 \leq  &
	 \sum_{e \in E } w(e) \LC g (f_e (y) ) - g (f_e (x) )   \RC 
	 =  \varphi _{G , g} (y) -\varphi _{G , g} (x) . 
\end{align*}
Hence $ L _{G , g } (x) \subset \partial \varphi _{G , g} (x ) $ holds and
it is enough to check the maximality of $ L _{G , g } $.

Let $ \{ a_e  \} _{e\in E}$ is a vector in $\R ^ E $.
If $a _e \geq 0$ for any $e \in E $, then
$\psi (x) := \sum_{e \in E } a _e w (e)  f_e ( x )  $ is a proper continuous convex function
and its subdifferential coincides with
$\partial \psi (x) =\{  \sum_{e \in E } a _e w (e)  b_e ,~b _e \in \argmax _{b \in B_e } b \cdot x  \} $
since $D (f_e ) =\R ^V $.
Hence the maximality of subdifferential
implies that
\begin{equation}
\label{Maximal-01} 
x + \sum_{e \in E } a _e w (e)  b _e  = y ,~~~b _e \in \argmax _{b \in B_e } b \cdot x 
\end{equation}
possesses a unique solution $x \in \R ^V $ for each given $y \in \R ^V $.
Here we define a mapping on $\R ^E $ by 
$\Gamma : \{ a_e  \} _{e\in E} \to \{ g'(f_e (x)) \} _{e\in E}$,
where $x $ is a solution to \eqref{Maximal-01}.

Multiplying \eqref{Maximal-01} by $x$ and recalling $b _e \cdot x = f_e(x) \geq 0$, we have 
\begin{equation*}
  | x  | ^2 _{\ell ^{2}} \leq   | x  | _{\ell ^{2}}  | y  | _{\ell ^{2}} , 
\end{equation*}
 which yields 
\begin{equation*}
  | x  | _{\ell ^{\infty }} \leq   | x  | _{\ell ^{2}} \leq  | y  | _{\ell ^{2}} .
\end{equation*}
 Since $f_e (x ) \leq 2|x| _{\ell ^{\infty}}$, we can see that 
$\Gamma $ maps a closed convex set 
\begin{equation*}
K := \left\{  \{ a _e \} _{e\in E} \in \R ^E ;~ 0 \leq a _e \leq L := \max _{0\leq s \leq 2 |y| _{\ell ^2}} g'(s) \right\}
\end{equation*}
into itself.
Next, suppose that $a ^m_e \to a  _e $ as $m\to \infty $ for each $e \in E$.
Let  $x ^m \in \R ^V $ and $b ^m _e  \in \argmax _{b \in B_e } b \cdot x^m $ satisfy \eqref{Maximal-01} 
with $\{ a ^m_e \} _{e \in E }$, i.e., $x ^m  + \sum_{e \in E } a ^m_e w (e)  b ^m_e  = y $.
Testing this equation by $x^m $, we get 
$  | x ^m  | _{\ell ^{\infty }} \leq  | y  | _{\ell ^{2}} $.
Moreover, the definition of $B_e$ yields $ | b ^m _e | _{\ell ^{2}} \leq 2 $.
Hence we can extract convergent subsequences of  $\{ x ^m \} $ and $ \{ b ^m _e \} $.
Let their limit be $x ^{\infty} \in \R ^V $ and $b ^{\infty}_ {e} \in \R ^V $. 
Since  the demiclosedness of maximal monotone operator leads to $b ^{\infty}_ {e}  \in \partial f_e (x ^{\infty })$,
we can see that $x ^{\infty} $  is a solution to \eqref{Maximal-01} with $\{ a _e \} _{e \in E }$
 by taking the limit as $m \to \infty $,
which together with the continuity of $f_e $ and $g' $ assures the continuity of $\Gamma $ (remark that 
 the original sequences $\{ x ^m \} $ and $ \{ b ^m _e \} $ converge to 
 $x ^{\infty} $ and $b ^{\infty}_ {e} $, respectively, by the uniqueness of solution to \eqref{Maximal-01}).
Therefore
Brouwer's fix point theorem is  applicable to $\Gamma $ 
and the solvability of $x + L _{G, g } (x) \ni y $ can be obtain for any $y \in \R ^V$.
\end{proof}

Hypergraph $p$-Laplacian
is defined by $L _{G , g }$ with $g (s) = \frac{1}{p} s ^p  $ ($p \geq 1$)
in Proposition \ref{Maximal-monotone}:
\begin{equation*}
\varphi _{G , p} (x ) := \frac{1}{p}   \sum_{e \in E } w (e ) ( f_e (x) ) ^p ,
\end{equation*}
and
\begin{align*}
L _{G , p } (x) &:=
\partial \varphi _{G , p}  (x) =
 \sum_{e \in E } w(e) (f_e (x) ) ^{p-1} \partial   f_e (x) \\
&=\LD \sum_{e \in E } w(e) (f_e (x) ) ^{p-1} b_e ;~~~b_e \in \argmax _{b \in B_e } b \cdot x  \RD .
\end{align*}

\begin{remark}
\label{Usual-Garph} 
If $ p>1 $ and $G$ is a usual graph, i.e., each $e \in E$ contains two elements, 
$L _{G , p } (x) $ becomes a single-valued operator.
Indeed, since  $f_e (x) = | x (v_ i ) - x (v _ j )|$ when $e = \{ v _i , v_ j \}$, 
we get
\begin{equation*}
\varphi _{G ,p} (x) = \frac{1}{2p } \sum_{i , j = 1 }^{n}  w _{ij} |x_ i - x_ j | ^p ,
\end{equation*}
where 
$x  (v _ i ) $ is abbreviated to $x_ i$
and 
$w_ {ij} := w   ( \{ v _ i , v_ j \} )$ if $\{ v _ i , v_ j \}  \in E$
(i.e., $v _ i $ and $ v_ j$ are connected directly)
and 
$w_ {ij} := 0 $ if $\{ v _ i , v_ j \} \not \in E$ 
(i.e., $v _ i $ and $ v_ j$ are disconnected).
Clearly, 
this functional is totally differentiable except $p= 1$
and its subgradient coincides with its derivative
(see Ch. 1.2 of \cite{Bar}). 
Especially, calculating partial derivative for the case where $p=2$, 
we have 
\begin{align*}
\partial _{x_ i } 
\varphi _{G ,2} (x)
& =  \sum_{ j = 1 }^{n}  w _{ij} ( x_ i - x_ j )
= d_ i  x_ i -\sum _{ j = 1 }^{n}  w _{ij} x_ j \\
&= (- w_{i1} ,\ldots , d_ i -  w _{ii} , \ldots , w_{i n }) \cdot x , 
\end{align*}
where $ d_ i := \sum_{i=1}^{n} w _{ij} $ denoting 
the (weighted) number of vertex connected to $v_ i $.
Hence 
$L _{G ,2 } = \partial \varphi _{G , 2} $
coincides with  $ D - A $, 
where the square matrix $D =  \diag (d_1 ,\ldots , d_ n)$ and $A = (w _{ij})$ 
are called the (weighted)  degree matrix and the (weighted) adjacency matrix.

On the other hand, 
when $G $ is hypergraph, $L _{G , p } (x) $ possibly returns a set-value 
on $\bigcup _{e \in E } \bigcup _{u ,v  \in e } \{  x \in \R ^V ;~x(u) = x(v)\}$
(union of hyperplanes)
by the singularity of derivative of max function
even if $p > 1$. 
\end{remark}

\subsection{Poincar\'{e}-Wirtinger Type Inequality}
We here decompose the vertex set $V$ into ``connected component'' by the following manner:
define $S_1 \subset V$ by the set of elements connected with $v_1 $.
\begin{equation*}
S_1 := \LD v _ i \in V ;~
\begin{matrix}
 \exists  u _ 1 , \ldots u _ {k-1} \in V , ~ \exists e _1 , e _2 , \ldots e_k \in E \text{~s.t.~}\\
u_{j-1} , u _j \in e_j ~~\forall j=1, 2, \ldots , k, \text{ where } u_ 0 = v_1 \text{ and } u_k = v_ i  
\end{matrix}
\RD .
\end{equation*}
If $S _1 \subsetneq V$, let 
$i_2$ be the least index satisfying $v_ {i _ 2 } \not \in S_1$
and define
\begin{equation*}
S_2 := \LD v _ i \in V ;
\begin{matrix}
 \exists  u _ 1 , \ldots u _ {k-1} \in V , ~ \exists e _1 , e _2 , \ldots e_k \in E \text{~s.t.~}\\
u_{j-1} , u _j \in e_j ~~\forall j=1, 2, \ldots , k, \text{ where } u_ 0 = v_{i_2}   \text{ and } u_k = v_ i  
\end{matrix}
\RD .
\end{equation*}
Continue this task inductively until $S_1 \cup S_2 \cup \ldots \cup S_ l = V$ holds.
One can expect the ``heat'' is not delivered between two separated  components.
It is easy to see that
\begin{itemize}
\item $S _j \cap S_ k = \varnothing $ if $j \neq k$.

\item for any $e \in E $, 
there exists $j\in \{1 ,\ldots ,l\} $ such that  $e \subset S_j $ and $e \cap S_ k = \varnothing $ if $k \neq j$.
\end{itemize}
Here we define $\phi \in \R ^V $ by 
\begin{equation*}
\phi = \phi _{c_1 ,\ldots , c_l} = \sum_{j=1}^{l} c _ j 1 _{S_j},
\end{equation*}
where $c_1 ,\ldots , c_ l \in \R $  are some constants, 
that is to say, $\phi $ satisfies $\phi (v) = c _ j$ if $v \in S_ j$.
We next show that
all 0-eigenfunction of $L_{G, p}$ can be denoted by $\phi $ for some $c_1 ,\ldots , c_ l $.

\begin{theorem}
\label{0-eigen-Rev} 
Let $p \geq  1$. Then $x \in \R ^V $ satisfies $0 \in L_{G , p} (x)$ if and only if
$x = \phi _{c_1 ,\ldots , c_l}$
with some constant $c_1 ,  \ldots , c_l \in \R $.
Moreover, for every $x \in \R ^V $ and  $c_1 ,  \ldots , c_l \in \R $, it holds that 
\begin{equation}
\varphi _{G , p} (x + \phi _{c_1 ,\ldots , c_l})  = \varphi _{G , p} (x),~~~~
L_{G , p} (x + \phi _{c_1 ,\ldots , c_l})  = L_{G , p} (x).  
\label{0-eigen_00} 
\end{equation}
\end{theorem}

\begin{proof}
We recall that the definition of the subdifferential yields 
\begin{equation*}
0 \in L_{G , p} (x) = \partial \varphi _{G, p} (x) ~~\Leftrightarrow ~~
 \varphi _{G, p} (x) = \min _{y \in \R ^V } \varphi _{G, p} (y) = 0.
\end{equation*}
Since $ w (e)> 0 $ and $f_e \geq 0$, $ \varphi _{G, p} (x) =\frac{1}{p}  \sum_{e \in E } w (e) \LC f _{e} (x) \RC ^p  = 0$ 
implies that  $f _{e} (x) =  \max  _{u ,v \in e} |x(u) -x(v )| =0 $ holds for every $ e \in E $,
and vice versa.
Therefore, it follows from $0 \in L_{G , p} (x) $ that $x(u) = x(v)$ for every $e \in E $,
namely, 
$x = \phi _{c_1 ,\ldots , c_l}$ with some constants $c_1 ,  \ldots , c_l \in \R $.
Conversely, we can easily obtain $0 \in L_{G , p} ( \phi _{c_1 ,\ldots , c_l} )$, i.e.,
$ \varphi _{G, p} (\phi _{c_1 ,\ldots , c_l}) = 0 $ 
since
$f _{e} (\phi _{c_1 ,\ldots , c_l}) =   |c_ j - c_j | =0 $ holds for every $ e \in E $ satisfying $e \subset S_ j$
and for every $j = 1, \ldots, l$.

Since 
$(1 _u - 1 _v ) \cdot \phi _{c_1 ,\ldots , c_l}  = c_ j - c_ j = 0 $ holds for any $u ,v \in e  \subset S_ j$,
we obtain $b \cdot \phi  _{c_1 ,\ldots , c_l}= 0$ for any $b \in B_{e} = \conv  \{ 1 _u - 1 _v ; ~~u,v \in e \}$.
Hence we obtain 
\begin{equation}
\begin{split}
&f_e (x + \phi _{c_1 ,\ldots , c_l}) =f_e (x ),~~~b \cdot ( x + \phi _{c_1 ,\ldots , c_l}) =b \cdot  x \\
& \hspace{5cm} \forall b \in B_e ,~~~\forall e \in E ,~~~ \forall  x \in \R ^V , 
\end{split}
\label{aver-can} 
\end{equation}
which leads to \eqref{0-eigen_00}.
\end{proof}
\begin{remark}
When $p> 1$, then $ L_{G , p} ( \phi _{c_1 ,\ldots , c_l} ) $ returns
a single value. Indeed,
\begin{equation*}
L_{G , p} ( \phi _{c_1 ,\ldots , c_l} ) =\sum_{e \in E }
w (e) \LC f _{e} (\phi _{c_1 ,\ldots , c_l}) \RC ^{p-1} \partial f_e (\phi _{c_1 ,\ldots , c_l} ) = 0.
\end{equation*}
\end{remark}

We define $\overline{x} \in \R ^V $ by averaging of $x \in \R ^V $ with respect to each connected component $S _ k $,
which can be expected to be a stable state of a system according to Theorem \ref{0-eigen-Rev}:
\begin{equation}
\overline{x} := \phi _{c_1 ,\ldots , c_l},~~\text{where~}
c_ k = \frac{1}{{\#}S_ k }  \sum_{v \in S_k } x (v) ~~~~~k=1 ,2, \ldots ,l .
\label{Average-vec} 
\end{equation}
Here and henceforth,
${\#} S$ stands for the number of elements belonging to $S \subset V$.
Then we can obtain the Poincar\'{e}--Wirtinger type inequality:
\begin{theorem}
\label{PoincareIn} 
Let $q \in [1, \infty ]$ and $p \geq 1$. Then every $x \in \R ^V $ and
$y \in L_{G, p} (x)$ satisfy
\begin{equation}
|x - \overline{x} | ^p _{\ell  ^q } \leq C x \cdot y 
=p C   \varphi _{G , p } (x), 
\label{Poincare_00} 
\end{equation}
where
\begin{equation}
C = C _{G, p} 
:=
\frac{1 }{ \min _{e \in E } w(e)} \LC  \sum_{j=1 }^{l} {\# } S_j  ^{ 2 -\frac{1}{p} } \RC ^p .
\label{Poincare_01} 
\end{equation}
\end{theorem}

\begin{proof}
Let $u ,v \in S_j $ ($j=1 , \ldots , l$).
By the definition, there exist 
$v _ {j1} ,  \ldots ,$ $v_{j(k -1)} \in S _j $ 
and $e _{j1 }  ,\ldots ,e_{j k } \in E $ such that
$v _ {j(i-1)} , v_{ji } \in  e_{ji }$ for any $i =1 ,2 ,\ldots ,k$, 
where $v _ {j0} = u$ and $v _ {jk} = v$.
Then we obtain 
\begin{align*}
|x (u) - x(v)| &\leq \sum_{i=1}^{k} |x(v_{j(i-1)} ) - x(v_{ji} ) | \leq
\sum_{i=1}^{k} f _{e _{ji} } (x) \\
&  \leq \frac{ {\#} S_ j ^{1/p'}}{ \min _{e \in E } w (e) ^{1/p}} \LC \sum_{e \in E} w(e) ( f_e (x) ) ^p \RC ^{1/p} ,
\end{align*}
where 
$p' $ is the H\"{o}lder conjugate exponent:
\begin{equation}
\label{Holder}
    p'  := \begin{cases}
    ~~\DS \frac{p}{p-1} ~~&~~\text{if} ~p>1 , \\
    ~~\DS \infty ~~&~~\text{if} ~p=1 .
    \end{cases}
\end{equation}
Recalling \eqref{Average-vec}, we obtain
\begin{align*}
|x (u ) - \overline{x} (u) | 
&\leq 
\frac{1}{{\#} S_ j}  \sum_{v \in S_ j }  |x (u ) - x (v) | \\
&\leq \frac{ {\#} S_ j ^{1/p'}}{ \min _{e \in E } w (e) ^{1/p}} \LC \sum_{e \in E} w(e) ( f_e (x) ) ^p \RC ^{1/p}
\end{align*}
for each $u \in S _j$, which leads to \eqref{Poincare_00} with $q = 1$.
From the general inequality $| x  | _{\ell ^q} \leq |x | _{\ell ^1}$, 
we can derive  \eqref{Poincare_00} for every $q \in ( 1 , \infty ]$.
\end{proof}

\begin{remark}
\label{Rem2-3-1} 
Let $y _j \in L _{G ,p} (x _ j ) $ ($j=1,2$), then we can easily obtain 
\begin{equation}
\label{differ}
\begin{split}
&(y_ 1 -  y_2 ) \cdot (x_1 - x_2 ) \\
&\geq 
\sum_{e \in E} w(e) \LC ( f _e (x_1) ) ^{p-1}  - ( f _e (x_2) ) ^{p-1}  \RC \LC  f _e (x_1)  -  f _e (x_2)  \RC .
\end{split}
\end{equation}
However, it seems to be difficult to establish an estimate 
of $| x_1 - x_2 | _{\ell ^2}$ or  $| x_1  - \overline{ x_1}  - x_2 + \overline{ x_2}  | _{\ell ^2}$ 
from this  inequality.
For instance, let ${\#}V = 4$, $E = \{ V \}$,  $w \equiv 1$, and 
\begin{equation*}
x_1 =
\begin{pmatrix}
x_1 (v_1) \\
x_1 (v_2) \\
x_1 (v_3) \\
x_1 (v_4)
\end{pmatrix} =
\begin{pmatrix}
1 \\
a_1 \\
b_1 \\
-1
\end{pmatrix} ,
~~~
x_2 =
\begin{pmatrix}
x_2 (v_1) \\
x_2 (v_2) \\
x_2 (v_3) \\
x_2 (v_4)
\end{pmatrix} =
\begin{pmatrix}
1 \\
a_2 \\
b_2 \\
-1
\end{pmatrix} ,
\end{equation*}
where $ |a _j| , |b_j | < 1$ ($j=1,2$).
Since $\argmax _{v \in V} x _ 1(v) =\argmax _{v \in V} x _ 2(v) =v_1$
and  $\argmin _{v \in V} x _ 1(v) =\argmin _{v \in V} x _ 2(v) =v_4$,
we have $f_e (x_1) =  f_e (x_2) = (1 - (-1 )) = 2$ and 
\begin{equation*}
L_{G, p} (x_1 ) =
L_{G, p} (x_2 ) 
 =
\begin{pmatrix}
2^{p-1} \\
0 \\
0 \\
- 2^{p-1}
\end{pmatrix} .
\end{equation*}
Hence we obtain
$ ( L _{G , p } (x_1 ) -  L _{G , p } (x_2 ) ) \cdot (x_1 - x_2) = 0 $
although $a _j ,  b_j \in (-1 , 1)$ can be chosen arbitrarily
so that $x_1 \neq x_2 $ and $ \overline{x_1} = \overline{ x_2} $.
\end{remark}

\section{Heat Equation}

\subsection{Cauchy Problem}
In this section, we consider the ordinary differential equation associated with the hypergraph $p$-Laplacian
by using properties given above.
\begin{equation*}
\text{(C)}
\begin{cases}
~~\DS  \frac{d}{dt} x (t) + L_{G ,p }(x (t) ) \ni h (t)  ~~~&~~t \in (0,T), \\
~~x (0) = x_0 , &
\end{cases}
\end{equation*}
where 
$x: [0,T] \to  \R ^V $ is an unknown function and
$h: [0,T] \to  \R ^V$ is a given external force.
Henceforth, we  write  $x'(t) := \frac{d}{dt} x(t) $, 
\begin{align*}
&\DS L^{q} (0, T ; \R ^V) := \LD  h : [0,T ] \to \R ^V ;~~\int_{0}^{T}  |h(t) | ^q_{\ell ^2 } dt < \infty  \RD  , \\
&\DS W^{1, q} (0, T ; \R ^V) := \LD  h : [0,T ] \to \R ^V ;
			~~\int_{0}^{T} \LC  |h(t) | ^q_{\ell ^2 } +  |h' (t) | ^q_{\ell ^2 }\RC  dt < \infty  \RD  ,
\end{align*}
 for $ q \in [1 , \infty )$,
and $h \in L ^{\infty} (0,T ; \R ^V ) $ iff $\esssup _{t \in [0,T] } |h(t) | _{\ell ^2} <\infty $. 
 
Since the maximal monotonicity of $L_{G ,p }$ 
in the Hilbert space $\R ^V$ endowed with $\ell ^2$-norm is shown in Proposition \ref{Maximal-monotone},
the abstract theory by K\={o}mura \cite{K} (see also Theorem 3.6 and 3.7 in Br\'{e}zis \cite{Bre})
can be applied and it holds that 
\begin{theorem}
For every $x _0 \in \R ^V $ and $h \in L ^2 ( 0, T ; \R ^V)$,
\textrm{(C)} possesses a unique solution satisfying
\begin{align*}
& x \in W ^{1,2}(0,T ; \R ^V  ), \\
& t \mapsto \varphi _{G ,p} (x (t)) ~~\text{ is absolutely continuous on }[0,T], \\
&   \LC \int_{0}^{t} |x '(s) | ^2 _{\ell ^2} ds  \RC ^{1/2}
	\leq \LC \int_{0}^{t} |h (s) | ^2 _{\ell ^2} ds \RC ^{1/2} + \sqrt{ \varphi _{G, p} (x_0)}
~~~\forall  t \in [0,T].
\end{align*}
Moreover, if $ t_ 0  \in [0, T ) $ is a right-Lebesgue point of $h$,
$x $ is right-differentiable at $t_0 $ and  the right-derivative of $x$ denoted by $ \frac{d^+ x}{dt }$ satisfies 
\begin{equation}
 \frac{d^+ x}{dt } (t_ 0 ) = \LC h(t_0 + 0) - L_{G ,p} (x (t_0 )) \RC ^{\circ }  ,
\label{Right} 
\end{equation}
where $h(t_0 + 0 ) := \lim_{\tau \to + 0} \frac{1}{\tau } \int_{t_0 }^{t_0 +\tau } h(s) ds  $.
If $h \in W ^{1,1} (0,T ; \R ^V )$, the solution also fulfills
\begin{align*}
& t x '(t) \in L^{\infty } (0,T ; \R ^V  ), \\
& \left | \frac{d^+ x}{dt } (t)  \right | _{\ell ^2}
\leq  |h (t) | ^2 _{\ell ^2}
		+ \frac{1}{t}  |x_0 - \overline{x_0 }  |  _{\ell ^2}
		+ \frac{1}{t ^2 } \int_{0}^{t} s ^2  \left | h' (s)  \right | _{\ell ^2} ds \\
&\hspace{3mm} 		+ \frac{\sqrt{2}}{t} 
\LC \int_{0}^{t} s   \left | h' (s)  \right | _{\ell ^2} ds \RC ^{1/2}		
\LC |x_0 - \overline{x_0 }  |  _{\ell ^2} +  \int_{0}^{t} \left | h (s)  \right | _{\ell ^2} ds \RC ^{1/2}~~\forall  t \in (0,T).
\end{align*}
\end{theorem}

\begin{remark}
As for a significant property of nonlinear multivalued evolution equation,
we here give an example of solution which dose not belong to $C ^1$-class.
Let ${\#}V = 4$, $E = \{ V \}$, and $w \equiv 1$.
We solve (C) with $ p= 2$ and given  data
\begin{equation*}
h\equiv 0, \hspace{5mm} x_ 0 = \begin{pmatrix}
x_0(v_1) \\
x_0(v_2) \\
x_0(v_3) \\
x_0(v_4) 
\end{pmatrix}
=
\begin{pmatrix}
2 \\
1 \\
-1 \\
-2
\end{pmatrix} .
\end{equation*}
When $t$ is sufficiently small, 
the order of initial data is preserved by the continuity,
i.e.,  $x _4 ( t ) <  x _3 ( t ) < x_2 ( t ) < x _1 ( t ) $ holds
(here and henceforth, we write $ x _i (t)  := x (t) (v_i ) $).
Then 
\begin{equation*}
L_{G  , 2 } (x(t)) =    
f _e (x(t))
 \begin{pmatrix}
1 \\
0 \\
0 \\
-1 \\
\end{pmatrix}
= 
\begin{pmatrix}
x_1 (t) - x_4(t) \\
0 \\
0 \\
x_4 (t) - x_1(t) \\
\end{pmatrix}.
\end{equation*}
Hence (C) is equivalent to  
\begin{equation*}
 x_1 ' (t) =-  x_1 (t) + x_4(t),
 ~~~
  x_4 ' (t) =  x_1 (t) - x_4(t),
 ~~~
 x_2 ' (t) =  x_3 ' (t) =  0 ,
\end{equation*}
which yields 
\begin{equation*}
x_1 (t) = 2 e ^{-2t} ,~~x_2 (t) = 1 ,~~x_3 (t) = -1 ,~~
x_4 (t) = -2 e ^{-2t} 
\end{equation*}
until $t \leq t_0 = \frac{1}{2}  \log 2  $.

In order to see the behavior of solution 
after  $x_1 $ and $ x_4 $ touches $x_2 $ and $ x_3 $, respectively,
i.e., $t \geq t_0 $,
we have to specify the minimal section of
\begin{equation*}
   L_{G , 2} (x (t))  
= 
\LD (x _1 (t) -x_4(t)) \begin{pmatrix}
\lambda _1 \\
\lambda _2 \\
- \mu _1 \\
- \mu _2 
\end{pmatrix};~~
\begin{matrix}
\lambda _ 1 ,  \lambda _ 2 , \mu _ 1 ,  \mu _ 2 \geq 0 ,\\
\lambda _ 1 +  \lambda _ 2 =1 , \\
 \mu _ 1 +  \mu _ 2 =1.
\end{matrix}
 \RD .
\end{equation*}
Since 
$ \sqrt{\lambda ^2 _1 + \lambda ^2 _2+  \mu ^2 _1 + \mu ^2 _2 }  $
attains its minimum 
at $\lambda _ 1 =   \lambda _ 2 =  \mu _ 1 =  \mu _ 2  = 1/2$, 
we get 
\begin{equation*}
   L_{G , 2} ^{\circ } (x (t))  
=\frac{ (x _1 (t) -
x_4(t))}{2}  \begin{pmatrix}
1 \\
1 \\
- 1\\
- 1 
\end{pmatrix}
=\frac{ (x _2 (t) -
x_3(t))}{2}  \begin{pmatrix}
1 \\
1 \\
- 1\\
- 1 
\end{pmatrix}.
\end{equation*}
By \eqref{Right}, (C) with $t \geq t_0 $ implies
 \begin{center}
\begin{tabular}{ll}
$\DS x_1 ' (t) =- \frac{1}{2} ( x_1 (t) - x_4(t) ) , $
    & $\DS  x_2 ' (t) =- \frac{1}{2} ( x_2 (t) - x_3(t) ) , $  \\[5mm]
$\DS x_3 ' (t) = \frac{1}{2} ( x_2 (t) - x_3(t) ) , $ 
    & $\DS x_4 ' (t) = \frac{1}{2} ( x_1 (t) - x_4(t) ) .$  
\end{tabular}
\end{center}
Therefore the solution of (C) is 
\begin{align*}
x_1 (t) &=
\begin{cases}
~~\DS  2 e^{-2t} ~~&~~\text{ if  }t \leq  \frac{1}{2}  \log 2,\\
~~\DS  \sqrt{2}  e^{-t} ~~&~~\text{ if  }t \geq  \frac{1}{2} \log 2 ,
\end{cases}~~~
~~x_4(t) = -x_1(t) , \\
x_2 (t) &=
\begin{cases}
~~\DS  1 ~~&~~\text{ if  }t \leq  \frac{1}{2}  \log 2 ,\\
~~\DS  \sqrt{2}  e^{-t} ~~&~~\text{ if  }t \geq  \frac{1}{2}  \log 2  ,
\end{cases}~~~~~
x_3(t) = -x_2(t).
\end{align*}
Analogously, we can construct a solution to (C) for the same hypergraph as the above with $ p\neq 2$
only by replacing  $L_{G_, 2 } (x) = ( f_e (x) )  \partial f_e (x)$ with
 $L_{G_, p } (x) = ( f_e (x) ) ^{p-1} \partial f_e (x)$.  
\end{remark}

Multiply the equation of (C) by $1 _{S_k}$ ($k=1 ,2, \ldots l$),
we have 
\begin{equation*}
\frac{d}{dt} \sum_{v \in S _k} x (t) (v ) = \sum_{v \in S _k} h (t) (v ) ,
\end{equation*}
which yields the following identity: 
\begin{equation*}
 \overline{x}  (t)  = \overline{x_0} + \int_{0}^{t}  \overline{h} (s)  ds .
\end{equation*}
Viewing this mass conservation law derived from \eqref{aver-can} 
and recalling Theorem \ref{0-eigen-Rev} and Poincar\'{e}-Wirtinger's inequality \eqref{Poincare_00},
one may expect to deal with the hypergraph Laplacian $L_{G,p}$ and \eqref{Eq}
by treatments similar to those for   
the standard Laplacian  $- \Delta $ with homogeneous Neumann boundary condition
and parabolic equations governed by the Neumann Laplacian.

We here consider the large time behavior of solution to (C).
Let $h \equiv 0$, then $ \overline{x}  (t)  = \overline{x_0} $ holds for any $t > 0 $.
By \eqref{aver-can} and Theorem \ref{PoincareIn},
\begin{align*}
 &\sum_{e \in E } w(e) (f_e (x(t)) ) ^{p-1} b_e (t) \cdot (x(t) - \overline{x_0} ) \\
   = &
 \sum_{e \in E } w(e) (f_e (x(t)) ) ^{p} \geq  \frac{1}{C_{G ,p }} |x(t) -\overline{x_0} |^p _{\ell ^2} 
~~~~ ~~~\forall 
   b_e (t) \in \argmax _{b \in B_e } b \cdot x (t).
\end{align*}
Hence multiplying (C) by $x(t) - \overline{x_0}$, we have 
\begin{equation*}
\frac{1}{2} \frac{d}{dt} |x(t) - \overline{x_0}| ^2 _{\ell ^2} 
 +\frac{1}{C_{G ,p }}  |x(t) - \overline{x_0}| ^p _{\ell ^2} \leq 0 ,
\end{equation*}
which leads to
\begin{theorem}
\label{Decay}
Let $h\equiv 0$. Then the solution to (C) satisfies 
\begin{center}
\begin{tabular}{ll}
$\DS X(t) \leq  \LC X(0) ^{\frac{2-p}{2} }  
			-  \frac{(2-p) t }{ C_{G, p}}  \RC _+ ^{\frac{2}{2-p} }$
	& if $1 \leq p<2 $, \\[3mm]
$\DS X(t) \leq X(0) \exp \LC -  \frac{ 2t }{ C_{G, p}}  \RC$
  & if $p=2$, \\[3mm]
$\DS X(t) \leq   \LC X(0) ^{- \frac{p-2}{2} }  
				+  \frac{(p-2) t }{ C_{G, p}}   \RC ^{- \frac{2}{p-2} }$
  & if  $p>2 $,\\[3mm]  
\end{tabular}
\end{center}
where $X (t) := |x (t) - \overline{x_0} | ^2 _{\ell ^2}$, 
 $( s ) _+ := \max \{ s, 0\}$, and $C_{G,p}$ is a constant defined by \eqref{Poincare_01}.
\end{theorem}

\begin{remark}
Optimality of decay rate in Theorem \ref{Decay} can be easily obtained as follows.
Since $\overline{y}(u) = \overline{y}  (v) $ holds for any $y \in \R ^V $ if $ u , v \in e$,
we have 
\begin{align*}
f_e (x (t  ) ) & = \max _{u ,v \in e } |x(t) (u) -x(t) (v) | 
	= \max _{u ,v \in e } |x(t) (u) - \overline{x_0} (u) + \overline{x_0} (v) -  x(t) (v) | \\
& \leq  \sum_{u \in V}  |x(t) (u) - \overline{x_0} (u) | \leq  \sqrt{n} 
		| x(t)  - \overline{x_0} | _{\ell ^2} .
\end{align*}
Hence testing  (C) by $x (t) - \overline{x_ 0} $, we can get
\begin{align*}
\frac{1}{2} \frac{d}{dt} | x(t)  - \overline{x_0} | ^2 _{\ell ^2}
& = - p \varphi _{G, p} (x (t) )  \\
& \geq  -  p \LC  \#  E \RC
n^{p/2} \max _{e\in E } w(e) | x(t) - \overline{x_0}  | ^p _{\ell ^2} ,
\end{align*}
which yields the estimate of $X(t) = | x(t) - \overline{x_0}  | ^2 _{\ell ^2}$ from below.
\end{remark}

\subsection{Periodic Problem}
Next we consider the  following time-periodic problem:
\begin{equation*}
\text{(P)}
\begin{cases}
~~\DS  \frac{d}{dt} x (t) + L_{G ,p }(x (t) ) \ni h (t)  ~~~&~~t \in (0,T), \\
~~x (0) = x (T) . &
\end{cases}
\end{equation*}
Note that the abstract result can not be applied since  $\varphi _{G ,p }$ is not coercive.
Multiplying (P) by $1 _{S_k}$, integrating over $[0,T]$ and using the periodicity,
we have 
\begin{equation}
\int_{0}^{T}   \overline{h} (t) dt =0
\label{A-Per} 
\end{equation}
as a necessary condition of existence of solution. 
Recall \eqref{Holder}, i.e., $p' $ stands for the the H\"{o}lder conjugate of $p$.

\begin{theorem}
\label{Per-Pr}
Let $h \in L^{\bar{p'}} ( 0,T ; \R ^V )$ with
 $\bar{p'} :=\max \{ 2 , p' \}$ and assume \eqref{A-Per}.
Then (P) possesses at least one solution $x \in W ^{1,2}(0,T ; \R ^V  )$. 
\end{theorem}

\begin{proof}
We first deal with the following approximation problem:
\begin{equation*}
\text{(P)}_{\varepsilon }
\begin{cases}
~~\DS  \frac{d}{dt} x _{\varepsilon }(t)
		+ \varepsilon x_{\varepsilon } (t)+ L_{G ,p }(x _{\varepsilon }(t) ) \ni h (t)  ~~~&~~t \in (0,T), \\
~~x _{\varepsilon } (0) = x _{\varepsilon }(T) . &
\end{cases}
\end{equation*}
Remark that the main term of (P)$_{\varepsilon }$ coincides with the subdifferential of
\begin{equation*}
\varphi ^{\varepsilon }_{G , p } (x) :=
\frac{\varepsilon }{2} |x | ^2_{\ell ^2} + \varphi _{G , p } (x) ,
\end{equation*}
i.e.,  $\partial \varphi ^{\varepsilon }_{G , p } (x) = \varepsilon x + L_{G ,p } (x)$.
Since $\varphi ^{\varepsilon }_{G , p } $ is coercive,
 (P)$_{\varepsilon }$ possesses a unique periodic solution $x _{\varepsilon }$
$x _{\varepsilon } \in W^{1,2} (0,T ; \R ^V)$
for any given $h \in L^2 (0,T ;\R ^V )$ (see Corollary 3.4 of \cite{Bre}).

Testing  (P)$_{\varepsilon }$ by $1 _{S _k}$ ($k=1,\ldots ,l$) and integrating over $[0, T]$,
we get 
\begin{equation*}
\sum_{v \in S_k } \int_{0}^{T} x_{\varepsilon }(t) (v) dt=0  
\end{equation*}
by the condition $x _{\varepsilon } (0) =  x _{\varepsilon } (T)  $ and \eqref{A-Per}.
Then the continuity of $x _{\varepsilon } $ implies that
there exists some $t_ k \in [ 0,T ]$ such that $\sum_{v \in S_k } x_{\varepsilon }(t_k) (v) =0  $.
Multiplying (P)$_{\varepsilon }$ by $1 _{S _k}$ again,
we obtain 
\begin{equation*}
\frac{d}{dt }  \LC \sum_{v \in S_k }  x_{\varepsilon }(t) (v)  \RC +
  \varepsilon \LC \sum_{v \in S_k } x_{\varepsilon }(t) (v)  \RC =    
  \LC \sum_{v \in S_k } h (t) (v) \RC ,
\end{equation*}
which leads to 
\begin{equation*}
  \LC \sum_{v \in S_k }  x_{\varepsilon }(t) (v) \RC =    
\int_{t_ k}^{t } e ^{ -\varepsilon (t -s)}
  \LC \sum_{v \in S_k } h (s) (v) \RC ds 
\end{equation*}
and 
\begin{equation}
| \overline{ x_{\varepsilon } } (t) | _{\ell ^q} \leq 
| \overline{ x_{\varepsilon } } (t) | _{\ell ^1} \leq 
\int_{0}^{T} | \overline{ h } (s) | _{\ell ^1} ds
\label{Per-01} 
\end{equation}
for any $q \in [1,\infty ]$ and $t \in [0 ,  T]$.

Multiplying (P)$_{\varepsilon }$ by $x_{\varepsilon }$ and integrating over $[0, T]$,
we have
\begin{align*}
& \varepsilon \int_{0}^{T} |x_{\varepsilon } (t) | ^2_{\ell ^2} dt 
+ p\int_{0}^{T} \varphi  _{G ,p } ( x_{\varepsilon } (t) ) dt \\
& \leq 
\int_{0}^{T} |h(t)| _{\ell ^2}  |x_{\varepsilon }(t)| _{\ell ^2}  dt \\
& \leq 
\LC \int_{0}^{T} |h(t)| ^{p' } _{\ell ^2}  dt \RC ^{1 /p'}
	\LB 
	\LC \int_{0}^{T}  |x_{\varepsilon }(t) - \overline{x_{\varepsilon }}  (t) | ^p _{\ell ^2}  dt \RC ^{1 /p} 
	+ 
	\LC \int_{0}^{T}  | \overline{x_{\varepsilon }}  (t) | ^p _{\ell ^2}  dt \RC ^{1 /p} 
\RB .
\end{align*}
From this estimate together with  \eqref{Poincare_00} and \eqref{Per-01},
we can derive 
\begin{equation}
 \varepsilon \int_{0}^{T} |x_{\varepsilon } (t) | ^2_{\ell ^2} dt 
+ \int_{0}^{T} |x_{\varepsilon } (t) | ^p_{\ell ^2} dt \leq C,
\label{Per-02} 
\end{equation}
where $C$ is some general constant independent of $\varepsilon  \in ( 0, 1] $.
Let $t _ 0 \in [0,T ]$ attain the minimum of $t \mapsto  |x_{\varepsilon } (t) | _{\ell ^2} $.  
Clearly $ |x_{\varepsilon } (t_ 0 ) | _{\ell ^2}  \leq C $ by \eqref{Per-02}.

Testing (P)$_{\varepsilon }$ by $x'_{\varepsilon } $,
we get
\begin{equation}
\int_{0}^{T} \left | x'_{\varepsilon } (t) \right | ^2_{\ell ^2} dt 
\leq 
\int_{0}^{T}  | h(t)  | ^2_{\ell ^2} dt .
\label{Per-03} 
\end{equation}
This immediately yields
\begin{equation}
\sup _{0\leq t \leq T}  |  x_{\varepsilon } (t)  | _{\ell ^2}  
\leq  C
\label{Per-04} 
\end{equation}
by  $ |x_{\varepsilon } (t_ 0 ) | _{\ell ^2}  \leq C $ and 
\begin{equation}
\int_{0}^{T}  | y_ \varepsilon (t) | ^2_{\ell ^2} dt 
\leq C ,
\label{Per-05} 
\end{equation}
where
 $y_ \varepsilon $ is the section of $ L _{G,p } (x_{\varepsilon })$ satisfying (P)$_{\varepsilon }$,
i.e., $x'_{\varepsilon } (t) + \varepsilon  x_{\varepsilon } (t) + y_{\varepsilon } (t) =h(t)$
and $ y_{\varepsilon } (t) \in L _{G,p } (x_{\varepsilon } (t))$
for a.e. $t \in (0,T )$.

By \eqref{Per-03} and \eqref{Per-04},
we can apply Ascoli-Arzela's theorem
and extract a subsequence (we omit relabeling) which strongly converges in $C([0, T ]; \R ^V)$.
Let $x$ be its limit, which evidently fulfills the periodic condition.
Then 
\eqref{Per-02} yields 
\begin{equation*}
| \varepsilon  x_{\varepsilon }  | _{L^2 (0,T ; \R ^V )} \leq  \sqrt{\varepsilon} C \to 0
\end{equation*}
and \eqref{Per-03} leads to 
\begin{equation*}
x'_{\varepsilon } \rightharpoonup x' ~~~\text{  weakly in } L^2(0,T ; \R ^V). 
\end{equation*}
Moreover, \eqref{Per-05} implies that
$\{ y_{\varepsilon }\} _{\varepsilon >0}$
also possesses a subsequence which weakly converges in  $L^2 (0, T ; \R ^V)$.
Thanks to the maximal monotonicity of $L_{G ,p}$,
its limit $y \in L^2 (0, T ; \R ^V)$ satisfies $y \in L_{G , p} (x )$,
whence it follows Theorem \ref{Per-Pr}
\end{proof}

Although the difference of two solutions can be hardly estimated
(see Remark \ref{Rem2-3-1}),
we can show  the uniqueness of periodic solution by virtue of Theorem 5 in \cite{H}:    
\begin{theorem}
\label{Uni} 
Let $x_ 1 , x_2 $ be two solutions to (P) with the same given $h $,
then there exists some constant $\gamma  \in \R ^V  $ such that $x _1 = x_2  + \gamma  $.  
\end{theorem}
\begin{remark}
We can easily see that  if $u$ is a solution to (P)
then 
$u + \phi _{c_1 ,c_2 ,\ldots ,c_k}$ also satisfies (P) for any $c_1 ,c_2 ,\ldots ,c_k \in \R $
and  obtain $f _e (x_1 (t)) =  f _e (x_2 (t)) $ by \eqref{differ}.
Still,  $\gamma  = \phi _{c_1 ,c_2 ,\ldots ,c_k}$ dose  not necessarily hold in Theorem \ref{Uni}.
For instance,
let 
$\alpha >0 $, $ \beta \geq  0 $ and 
\begin{align*}
h (t) =
\begin{pmatrix}
2 \alpha \exp \LC 2 (t-\frac{T}{2} ) \RC  + 2 \beta \\
0 \\
0 \\
- 2 \alpha  \exp \LC 2 (t-\frac{T}{2} ) \RC   - 2 \beta 
\end{pmatrix},
\end{align*}
then 
\begin{align*}
x (t) =
\begin{pmatrix}
 \alpha \cosh \LC 2 (t-\frac{T}{2} ) \RC +\beta \\
a \\
b \\
- \alpha \cosh \LC 2 (t-\frac{T}{2} ) \RC - \beta 
\end{pmatrix}
\end{align*}
becomes  a periodic solution to (P) with
${\#}V = 4$, $E = \{ V \}$, $w \equiv 1$, and $p=2$ 
for arbitrary fixed $ a , b \in  (- \alpha - \beta ,  \alpha + \beta )$. 
\end{remark}

\section*{Conclusion}

In this article, we study the hypergraph $p$-Laplacian
from the viewpoint of nonlinear analysis
and find the lack of coerciveness and the Poincar\'{e}--Wirtinger type inequality
for this operator.
We can see some validity of these tools 
in the treating Cauchy problem and time-periodic problem of 
the evolution equation governed by the  hypergraph $p$-Laplacian.

Interestingly, 
the multiplicity of $L_{G, p }$ implies that 
the ODE describes the diffusion of ``heat'' from the vertex with  maximum
to minimum belonging to the same hyperedge
and 
the vertices with middle value 
halt until the maximum or minimum touches.
This property might suggest a new PDE model describing competition of two groups.
Namely, we can expect hypergraph Laplacian in reaction-diffusion system
describes the effect of aid/replenishment/assistance
to injured/suffering members (vertices) from others in each group (hyperedge).


\section*{Acknowledgment}
M. Ikeda is supported by JST CREST Grant Number JPMJCR1913, 
Japan and Grant-in-Aid for Young Scientists Research (No.19K14581), Japan Society for the Promotion of Science.
S. Uchida is supported by the Fund for the Promotion of Joint International Research 
(Fostering Joint International Research (B)) (No.18KK0073), Japan Society for the Promotion of Science.

%
\address{
Masahiro Ikeda\\
Department of Mathematics, \\
Faculty of Science and Technology, \\   
Keio University, \\
3-14-1 Hiyoshi Kohoku-ku, Yokohama, \\
223-8522, JAPAN/\\
Center for Advanced Intelligence \\
Project, RIKEN, Tokyo, \\
103-0027, JAPAN.
}
{masahiro.ikeda@keio.jp/\\
masahiro.ikeda@riken.jp}
%
%
\address{
Shun Uchida\\
Department of Integrated Science and Technology, \\
Faculty of Science and Technology, \\ 
Oita University,\\
700 Dannoharu, Oita City, Oita Pref., \\
 870-1192, JAPAN.
}
{shunuchida@oita-u.ac.jp}
\end{document}